\documentclass[11pt,reqno]{amsart}

\usepackage{amsmath, amsfonts, amsthm, amssymb,mathrsfs,hyperref}

\newtheorem{Theorem}{Theorem}[section]
\newtheorem{Lemma}{Lemma}[section]

\theoremstyle{definition}

\theoremstyle{remark}
\newtheorem{Remark}{Remark}[section]

\numberwithin{equation}{section}

\renewcommand{\u}{{\bf u}}

\newcommand{\R}{{\mathbb R}}
\newcommand{\Dv}{{\rm div}}

\newcommand{\m}{{\bf m}}

\newcommand{\T}{{\mathcal T}}
\def\f{\frac}
\renewcommand{\O}{\Omega}
\renewcommand{\T}{\mathbb T^N}

\def\hf1{^\f{1}{1-\xi^2}}

\def\be{\begin{equation}}
\def\en{\end{equation}}
\def\bs{\begin{split}}
\def\es{\end{split}}

\author[R.M. Chen]{Robin Ming Chen}
\address{Department of Mathematics, University of Pittsburgh, Pittsburgh, PA 15260.}
\email{mingchen@pitt.edu}

\author[C. Yu]{Cheng Yu}
\address{Department of Mathematics, University of Texas, Austin, TX 78712.}
\email{yucheng@math.utexas.edu}

\title[Energy conservation for inhomogeneous Euler equations]
{Onsager's energy conservation for inhomogeneous Euler equations}

\subjclass[2000]{}

\date{\today}

\begin{document}
\begin{abstract} This paper addresses the problem of energy conservation for the two- and three-dimensional density-dependent Euler equations. Two types of sufficient conditions on the regularity of solutions are provided to ensure the conservation of total kinetic energy on the entire time interval including the initial time. The first class of data assumes integrability on the spatial gradient of the density, and hence covers the classical result of Constantin-E-Titi \cite{CET} for the homogeneous Euler equations. The other type of data imposes extra time Besov regularity on the velocity profile, and the corresponding result can be applied to deal with a wide class of rough density profiles.
\end{abstract}

\maketitle
\section{Introduction}

The theory of hydrodynamics is a fascinating subject, with a long history in both pure and applied mathematics. The first comprehensive model was proposed by Euler in the 1750s, with the most notable variant introduced almost a century later by Navier and Stokes to allow for viscous effects. Besides being able to describe smooth flow motion, the Euler equations also model a wide range of fluids with singularities or limited smoothness, for instance flows with point vortices, vortex sheets, and turbulent flows in the limit of vanishing viscosity. Such configurations naturally lead to considering the Euler equations in a {\it weak} sense.

Given that the solution to the Euler equations is sufficiently smooth, say, for e.g.,  $C^1$, it is easy to see that the total kinetic energy of the flow is conserved. On the other hand it had long been observed experimentally (but mathematically still open) that  anomalous dissipation of energy -- energy dissipation independent of viscosity -- persists in fully developed turbulent flow. Therefore it is reasonable to expect the existence of weak solutions to the Euler equations which do not conserve energy; see Scheffer \cite{Scheffer}, Shnirelman \cite{Shnirelman} and De Lellis--Sz\'ekelyhidi \cite{DS}. This possibility goes by the name of the ``Onsager conjecture" \cite{O}: non-conservation of energy in the three-dimensional Euler equations would be related to the loss of regularity. Specifically, Onsager conjectured that every weak solution to the Euler equations with H\"older continuity exponent $\alpha > {1\over 3}$ conserves energy; and anomalous dissipation of energy occurs when $\alpha < {1\over 3}$. See \cite{DS13,ES06,F95,R03} for reviews and further discussions.

The first part of the conjecture was proved by Eying \cite{Ey}, Constantin-E-Titi \cite{CET}, Duchon-Robert \cite{DR}, Cheskidov et al. \cite{CCFS}, among others. The sharpest result is given in \cite{CCFS}, where the conservation of energy was proved in the Besov space setting, and the results allow for possible failure of energy conservation in the endpoint $\alpha = {1\over 3}$ case.

The development toward the other direction of the conjecture is more recent. The rigorous mathematical work establishing existence of dissipative weak Euler solutions of the type conjectured by Onsager began with the pioneering work of DeLellis-Sz\'ekelyhidi \cite{DS13b,DS14} based on the convex integration approach, and has since culminated in constructions of solutions up to the critical 1/3 regularity \cite{BDSV,I,I2}.

In this paper we consider the energy conservation for the weak solutions of the {\it inhomogeneous} Euler equations, namely
\begin{equation}
\begin{split}
\label{Euler}
&\rho_t+\Dv(\rho\u)=0\\
&(\rho\u)_t+\Dv(\rho\u\otimes\u)+\nabla P=0\\
&\Dv\u=0,
\end{split}
\end{equation}
with initial data
\begin{equation}
\label{initial data}
\rho|_{t=0}=\rho_0(x),\;\;\;\;\;\rho\u|_{t=0}=\m_0(x)=\rho_0\u_0,
\end{equation}
where $P$ denotes the pressure, $\rho\ge 0$ is the density of fluid, $\u$ stands for the velocity of fluid. For the sake of simplicity we will consider the periodic setting $x \in \mathbb{T}^N$ with $N = 2, 3$.
%case of bounded domain with periodic boundary conditions in $\R^N$, namely $\O=\mathbb{T}^N$, $N=2, 3$.
Here we define $\u_0=0$ on the set $\{x:\ \rho_0(x)=0\}.$

Density fluctuations are widely present in turbulent flows. They may arise from non-uniform species concentrations, temperature variation or pressure. In geo-fluids, for instance, where external forces are present (e.g. gravity), density stratification can be caused by temperature and salinity gradients (ocean) or moisture effects (atmosphere).  Also, a strong density inhomogeneity can be induced by the mixing of different-density species. A canonical example is the Rayleigh-Taylor instability of an interface between two fluids of different densities.

Mathematically, having a variable density changes the regularity structure substantially, and poses additional challenges compared to the much more extensively studied homogeneous flows. The vorticity equation becomes
\begin{equation*}
\omega_t + \u \cdot \omega = (\omega \cdot \nabla) \u + {\nabla \rho \times \nabla P \over \rho^2},
\end{equation*}
where in 2D, the cross product is understood as $\nabla^\perp \rho \cdot \nabla P$. One sees that fluid particles representing different instantaneous density respond very differently to
pressure gradients. Vorticity can be generated from non-aligned density and pressure gradients (the baroclinic torque). Moreover, the regularity of the density gradient does not propagate from the initial data. All of these indicate possible energy fluctuation coming from the loss of smoothness of the density.

On the other hand, the roughness of the density can be ``traded off" by assuming more regularity of the velocity field. The general strategy to prove energy conservation is to mollify the momentum equation and then test it against some suitable velocity-type test function. When passing to the limit as the mollification parameter tends to zero, the commutator estimates are required for treating the nonlinear terms. Compared with the homogeneous equations, the momentum equation in \eqref{Euler} contains a nonlinear term $\rho\u$ in the time derivative, and hence needs a commutator estimate (in time).

Leslie-Shvydkoy \cite{LS} managed to avoid this time commutator estimate by using the momentum $\m = \rho\u$ as the unknown variable instead of the velocity $\u$. This way to obtain the energy conservation one needs to choose the test function to be ${\m \over \rho}$ in a regularized form. The price to pay is that (i) the divergence-free structure of the test function is no longer valid, and hence additional assumptions on the pressure has to be made; (ii) the density must be bounded away from zero.

Feireisl et al \cite{FGSW} took a more direct approach by assuming Besov regularity also in time to pertain the divergence-free property of the test function and to allow for the existence of vacuum in the system. Their method can also be applied to treat the case of isentropic compressible flows. What was proved there is a stronger result, namely a local energy equality in the sense of distribution, under a different set of regularity assumptions than those of \cite{LS}. However it is still conceivable that the energy might fluctuate in time in a set of times of measure zero, unless additional smoothness conditions are imposed.

The objective of this paper is to continue addressing the relation between the energy conservation and the degree of regularity of the solutions for system \eqref{Euler}. In particular we provide sufficient conditions on the regularity of solutions to ensure the conservation of the total energy. Our approach is in the spirit of Constantin-E-Titi \cite{CET}, with additional care to the density term $\rho$. We choose to work with the unknowns $\u$ and $\rho$, like in Feireisl et al \cite{FGSW}. The main differences between our method and the one in \cite{FGSW}, which also constitutes the main contribution of this paper, are explained in the paragraphs below. Note that many of the ideas have been successfully developed to the case of {\it compressible} Navier-Stokes equations \cite{VY}, wherein the main purpose was to derive a priori estimates rather than energy equality. Further development of the method does lead to the energy conservation of such flows \cite{Yu2}. Here we see that this tool also works well for incompressible flows.

\subsubsection*{Time regularity of $\rho$}
The first difficulty comes from the fact that the commutator estimates for the nonlinear time derivative $(\rho\u)_t$ naturally requires the time regularity of $\rho$. However from the mass conservation, this time regularity can be transferred to the spatial regularity of $\rho$. As is explained earlier, the regularity of $\nabla \rho$ does not propagate by the flow. Therefore imposing assumptions on $\nabla \rho$ seems to be a reasonable choice for the energy conservation. In particular, doing so allows one to avoid assuming additional time regularity on the velocity field $\u$, and hence can recover the classical result of Constantin-E-Titi \cite{CET}; see Remark \ref{rk_thm1}. Moreover, similar as in \cite{FGSW}, working with $\rho$ and $\u$ makes it possible to choose test functions with a divergence-free structure. As a consequence, our result applies to the case when vacuum is present.

\subsubsection*{Pressure} When the test function hits the pressure term in the momentum equation, it is natural to impose some conditions on the pressure in order to obtain the energy equality. In the constant density case, the pressure solves an elliptic equation
\begin{equation*}
-\Delta P = \sum_{i, j} \partial_{x_i}\partial_{x_j} (u_iu_j),
\end{equation*}
and thus the regularity of pressure can be inferred from the velocity. However for inhomogeneous flows, such a conclusion does not hold unless the density is sufficiently regular. Therefore in the density-dependent case it is common to assume certain regularity condition on the pressure; see \cite{FGSW,LS}. Here we introduce a special test function that is divergence-free to remove the pressure term completely, so that we do not need any assumption on the pressure.

\subsubsection*{Continuity at $t = 0$}
In the homogeneous case, the velocity field $\u$ is continuous at time $t=0$, and hence the energy conservation holds all the way up to the initial time. This is not necessarily true for inhomogeneous flows. However from the energy point of view, the analogous requirement would be that $\sqrt{\rho} \u$ is continuous in the strong topology at $t = 0$. This can be done by studying the continuity of $\rho\u$ and $\sqrt{\rho}$ at $t=0$. Using the equations they satisfy, one can show that they are continuous at $t=0$ in the weak topology, which, together with the initial regularity on $\u$, is enough to conclude the continuity of $\sqrt{\rho}\u$ in the strong topology. Finally to continue the energy conservation up to the initial time, we introduce a special type of test function which originally vanishes near $t=0$, but will later be extended past $t=0$ into $t<0$. %This technique has again been used in the study of compressible flows \cite{VY,Yu2}.

\vskip0.3cm

\subsection{Main results}

We provide two types of sufficient regularity conditions on the weak solution of \eqref{Euler} that ensure the conservation of the energy. The first one requires the control of $\nabla\rho$ to avoid additional time regularity assumption on $\rho$. The definition of the Besov spaces $B^{\alpha,\infty}_p$ is given in Section \ref{Sec_Besov}.

\begin{Theorem}\label{thm_main1}
Let $(\rho,\u)$ be a weak solution of \eqref{Euler}-\eqref{initial data} in the sense of distributions. Assume
\begin{equation}
\label{condition for first theorem}
\rho\in L^{\infty}(0,T;\T)\cap L^{p}(0,T;W^{1,p}(\T)),\;\;\u \in L^q(0,T;B_q^{\alpha,\infty}(\T)),
\end{equation}
for any $\frac{1}{p}+\frac{3}{q}\leq 1$ and $\alpha>\frac{1}{3},$
\begin{equation}
\label{condition for kinetic energy}\sqrt{\rho}\u \in L^{\infty}(0,T;L^2(\T))
\end{equation}
 and $$\u_0\in L^2(\T).$$ Then the energy $$E(t)=\int_{\T}\rho|\u|^2\,dx$$
is conserved, that is, we have $E(t)=E(0),$
 for any $t\in [0,T].$
\end{Theorem}
\begin{Remark}\label{rk_thm1}
In the case of constant density, condition \eqref{condition for first theorem} recovers the classical result of Constantin-E-Titi \cite{CET} by taking $p = \infty$ and $q = 3$.
\end{Remark}

Our second theorem treats the case when density can experience large fluctuation. This covers a rather wide range of situations including flows with mixing layers and vortex sheets. In this case we need to impose extra time regularity on the velocity field $\u$ to compensate for the roughness of $\rho$.
\begin{Theorem}\label{thm_main2}
Let $(\rho,\u)$ be a weak solution of \eqref{Euler}-\eqref{initial data} in the sense of distributions. Assume
\begin{equation}
\label{condition for the second theorem}
\rho\in L^{\infty}(0,T;L^{\infty}(\T)),\;\;\u \in B_p^{\beta,\infty}(0,T;B_q^{\alpha,\infty}(\T)),
\end{equation}
where $p,q\geq 3$, $2\alpha + \beta > 1$
and  $\alpha+2\beta> 1,$ \begin{equation*}
\sqrt{\rho}\u \in L^{\infty}(0,T;L^2(\T))
\end{equation*} and $$\u_0\in L^2(\T).$$ Then the energy $$E(t)=\int_{\T}\rho|\u|^2\,dx$$
is conserved, that is, we have $E(t)=E(0),$ for any $t\in [0,T].$
\end{Theorem}
\begin{Remark}\label{rk_thm2}
Again in this theorem we do not pose any time regularity of $\rho$ in the Besov spaces. To see this, note that when dealing with the commutator estimates for $(\rho\u)_t$, an integration by parts places the time derivative on the test function which involves only the velocity $\u$. Hence assuming more time regularity on $\u$ would lead to the desired convergence.
\end{Remark}

\vskip0.3cm
\bigskip

\section{Besov Space and a commutator lemma}\label{Sec_Besov}
In this section we briefly recall some properties of the Besov space $B_p^{\alpha,\infty}( \O)$, and prove a key estimate which is similar to the commutator lemma in \cite{F04,L}.

For $0 < \alpha < 1$ and $1\le p \le \infty$, we define the Besov space to be the set of all functions equipped with the following norm
\begin{equation}
\label{Besov norm}
\|w\|_{B_p^{\alpha,\infty}( \O)}:=\|w\|_{L^p(\O)}+\sup_{\xi\in \O}\frac{\|w(\cdot+\xi)-w\|_{L^p(\O\cap (\O-\xi)}}{|\xi|^{\alpha}},
\end{equation}
where $\O-\xi=\{x-\xi:x\in\O\}.$

Next we define $$f^{\varepsilon}(t,x) := \eta_{\varepsilon}* f(t,x),\;\;\; t>\varepsilon,$$
where $\eta_{\varepsilon}=\frac{1}{\varepsilon^{N+1}}\eta(\frac{t}{\varepsilon},\frac{x}{\varepsilon})$, and $\eta(t,x)\geq 0$ is a smooth even function compactly supported in the space-time ball of radius $1$, and with integral equal to $1$.

Now we are ready to recall some classical properties of the Besov space as follows:
\begin{equation}
\label{convergence in Besov space}
\|w^{\varepsilon}-w\|_{L^p(\O)}\leq C\varepsilon^{\alpha}\|w\|_{B_p^{\alpha,\infty}( \O)}
\end{equation}
and
\begin{equation}
\label{gradient property in Besov space}
\|\nabla w^{\varepsilon}\|_{L^p(\O)}\leq C\varepsilon^{\alpha-1}\|w\|_{B_p^{\alpha,\infty}( \O)}.
\end{equation}

We will rely on the following lemma  which was proved in \cite{F04, L}. The statement of the result we adopt here is in the spirit of \cite{LV}.
 \begin{Lemma}
 \label{Lions's lemma}
Let $\partial$ be a partial derivative in space or time. Let $f$, $\partial f \in L^p(\R^+\times \O)$, $g\in L^q (\R^+\times \O)$ with $1\leq p,\,q \leq \infty,$ and $\frac{1}{p}+\frac{1}{q}\leq 1. $ Then, we have
$$\left\|[\partial(fg)]^{\varepsilon}-\partial(f\,g^{\varepsilon}) \right\|_{L^r(\R^+\times\O)}\leq C\|\partial f\|_{L^p(\R^+\times\O)}\|g\|_{L^q(\R^+\times\O)}$$
for some constant $C>0$ independent of $\varepsilon$, $f$ and $g$, and with $\frac{1}{r}=\frac{1}{p}+\frac{1}{q}$. In addition,
$$[\partial(fg)]^{\varepsilon}-\partial(f\,g^{\varepsilon}) \to 0\quad\text{ in } L^r(\R^+\times\O)$$
as $\varepsilon\to 0$ if $r<\infty.$
 \end{Lemma}

For the purpose of this paper, we need to extend the above lemma to the Besov space.
 \begin{Lemma}
 \label{key lemma}
 Let $f \in B_p^{\alpha,\infty}( \O)$, $g\in L^q (\O)$ with $1\leq p,\,q \leq \infty.$  Then, we have
$$\|(fg)^{\varepsilon}-f\,g^{\varepsilon}\|_{L^r(\O)}\leq C\| f\|_{B_p^{\alpha,\infty}( \O)}\|g\|_{L^q(\O)}$$
for some constant $C>0$ independent of  $f$ and $g$, and with $\frac{1}{r}=\frac{1}{p}+\frac{1}{q}$. In addition,
$$\|(fg)^{\varepsilon}-f\,g^{\varepsilon}\|_{ L^r(\O)}\leq C\varepsilon^{\alpha}\to 0$$
as $\varepsilon\to 0$ if $r<\infty.$
 \end{Lemma}
\begin{proof}
Considering
\begin{equation*}\begin{split}&
(fg)^{\varepsilon}-f\,g^{\varepsilon}=\int_{\O}\left(f(y)g(y)-f(x)g(y)\right)\eta_{\varepsilon}(|x-y|)\,dy
\\&=\int_{\O}g(y)\left(f(y)-f(x)\right)\eta_{\varepsilon}(|x-y|)\,dy
\\&=\int_{\O}g(x-z)\frac{f(x-z)-f(x)}{|z|^{\alpha}}\eta_{\varepsilon}(|z|)|z|^{\alpha}\,dz,
\end{split}
\end{equation*} for any $x\in \O$, we have
\begin{equation*}
\|(fg)^{\varepsilon}-f\,g^{\varepsilon}\|_{L^r(\O)}\leq C\|g\|_{L^q(\O)}\|f\|_{B_p^{\alpha,\infty}( \O)}\varepsilon^{\alpha}. \qedhere
\end{equation*}
\end{proof}

\vskip0.3cm
\section{Proof of Theorem \ref{thm_main1}}

We introduce a function $\Phi(t,x) =(\psi(t)\u^{\varepsilon})^{\varepsilon}$ as a test function for deriving the energy equality, where $\psi(t)\in  \mathfrak{D}(0,+\infty)$ is a test function,
with $\mathfrak{D}(0,+\infty)$ being the class of all smooth compactly supported functions on $(0,+\infty)$. Here we further remark that this function vanishes  near $t=0$.
However, later it is needed to extend the result for $\psi(t)\in  \mathfrak{D}(-1,+\infty)$ in order to recover the initial value of the energy.

Note that $\psi(t)$ is compactly supported in $(0,\infty).$ Hence $\Phi$ is a well-defined test function for $t\in (0,\infty)$ and for $\varepsilon$ small enough. Multiplying $\Phi$ on both sides of the second equation in \eqref{Euler}, one obtains
\begin{equation*}
\int_0^T\int_{\T}\Phi\left[(\rho\u)_t+\Dv(\rho\u\otimes\u)+\nabla P\right]\,dx\,dt=0,
\end{equation*}
which in turn yields
\begin{equation}
\label{weak formulation}
\int_0^T\int_{\T}\psi(t)\u^{\varepsilon}\left[ (\rho\u)_t+\Dv(\rho\u\otimes\u)+\nabla P\right]^{\varepsilon}\,dx\,dt=0,
\end{equation}
where we have used the fact that $\eta(-t,-x)=\eta(t,x).$

The first term in \eqref{weak formulation} can be computed as
\begin{equation}
\begin{split}
\label{the first term}
\int_0^T & \int_{\T}\psi(t)\u^{\varepsilon}\left((\rho\u)_t\right)^{\varepsilon}\,dx \\
&=\int_0^T\int_{\T}\psi(t)\left[\left((\rho\u)_t\right)^{\varepsilon}-(\rho\u^{\varepsilon})_t\right]\u^{\varepsilon}\,dx\,dt  +\int_0^T\int_{\T}\psi(t)(\rho\u^{\varepsilon})_t\u^{\varepsilon}\,dx\,dt
\\&=: A_{\varepsilon} +\int_0^T\int_{\T}\psi(t)\rho\partial_t{\frac{|\u^{\varepsilon}|^2}{2}}\,dx\,dt+\int_0^T\int_{\T}\psi(t)\rho_t|\u^{\varepsilon}|^2\,dx\,dt.
\end{split}
\end{equation}
Similarly, the second term in \eqref{weak formulation} can be treated as
\begin{equation}
\begin{split}
\label{the second term}
\int_0^T & \int_{\T}\psi(t) \u^{\varepsilon}\left(\Dv(\rho\u\otimes\u)\right)^{\varepsilon}\,dx\,dt \\
& =
\int_0^T\int_{\T}\psi(t)\left[\left(\Dv(\rho\u\otimes\u)\right)^{\varepsilon}-\Dv(\rho\u\otimes\u^{\varepsilon})\right]\u^{\varepsilon}\,dx\,dt \\
&\quad\ \ +\int_0^T\int_{\T}\psi(t)\Dv(\rho\u\otimes\u^{\varepsilon})\u^{\varepsilon}\,dx\,dt
\\&=: B_{\varepsilon}+\int_0^T\int_{\T}\psi(t)\rho\u\cdot\nabla\frac{|\u^{\varepsilon}|^2}{2}\,dx+\int_0^T\int_{\T}\psi(t)\Dv(\rho\u)|\u^{\varepsilon}|^2\,dx\,dt
\\&=B_{\varepsilon}-\int_0^T\int_{\T}\psi(t)\rho_t\frac{|\u^{\varepsilon}|^2}{2}\,dx,
%+\int_0^T\int_{\T}\psi(t)\Dv(\rho\u)|\u^{\varepsilon}|^2\,dx\,dt,
\end{split}
\end{equation}
where an integration by parts and the first equation of \eqref{Euler} are used to obtain the last equality.
%second term on the right-hand side. Similarly, using integration by parts and the divergence-free condition on $\u^\varepsilon$ we see that the third term in the second equality vanishes.

Note that from \eqref{condition for first theorem}, $\nabla\rho\in L^{\infty}(0,T;L^p(\T))$ and $\rho_t=-\u\cdot\nabla\rho$. The last term of the right-hand side in \eqref{the first term} and the second term of the right-hand side in \eqref{the second term} are well-defined. Thanks to  \eqref{weak formulation}, \eqref{the first term} and \eqref{the second term},  we have
\begin{equation}
\label{the first two terms}
-\int_0^T\int_{\T}\psi_t\frac{1}{2}\rho|\u^{\varepsilon}|^2\,dx\,dt+A_{\varepsilon}+B_{\varepsilon}=0.
\end{equation}

From \eqref{condition for first theorem} and \eqref{convergence in Besov space}, one obtains
\begin{equation}
\label{convegence first term}
\int_0^T\int_{\T}\frac{1}{2}\rho|\u^{\varepsilon}|^2\psi_t\,dx\,dt \longrightarrow \int_0^T\int_{\T}\frac{1}{2}\rho|\u|^2\psi_t\,dx\,dt\;\;\text{as }\;\varepsilon\to0.
\end{equation}

\vskip0.3cm

To derive the energy equality from \eqref{the first two terms} in the distributional sense,
we will apply Lemma \ref{Lions's lemma} to prove
\begin{equation}
\label{the rest goes to zero}
A_{\varepsilon}(t,x)\to 0\quad\text{ and }\quad B_{\varepsilon}(t,x)\to 0
\end{equation}
 as $\varepsilon$ goes to zero.

In fact, note from $\rho_t=-\u\cdot\nabla\rho$ and \eqref{condition for first theorem} that $\rho_t$ is bounded in
$L^{q}(0,T;L^{\kappa}(\T))$ where $\frac{1}{\kappa}=\frac{1}{p}+\frac{1}{q}.$
Thus, Lemma \ref{Lions's lemma} gives
\begin{equation*}
\begin{split}
|A_{\varepsilon}|&\leq\|\psi(t)\|_{L^{\infty}(0,T)}\int_0^T\int_{\T}\Big| \u^{\varepsilon}[\left((\rho\u)_t\right)^{\varepsilon}-(\rho\u^{\varepsilon})_t]\Big|\,dx\,dt
\\&\leq C\|\psi(t)\|_{L^{\infty}(0,T)}\|\rho_t\|_{L^{q}(0,T;L^{\kappa}(\T))}\|\u\|^2_{L^q(0,T;L^q(\T))},
\end{split}
\end{equation*}
where $\frac{3}{q}\leq 1$ and $\frac{3}{q}+\frac{1}{p}\leq 1.$ Moreover,  as $\varepsilon$ tends to zero, we have
\begin{equation*}
\label{term A goes to 0}
\begin{split}
A_{\varepsilon}\to 0.
\end{split}
\end{equation*}
% An argument similar to that given above
%in the control for $A_{\varepsilon}$ shows that the terms $B_{\varepsilon}$ converges  to zero, as $\varepsilon\to 0$.

We are not able to apply
Lemma \ref{Lions's lemma} to control $B_{\epsilon}$ directly because there is no estimate on $\nabla(\rho\u)$ in $L^p$. Instead, we will use Lemma \ref{key lemma}.
In fact,
\begin{equation*}
\begin{split}B_{\epsilon}&=
\int_0^T\int_{\T}\psi(t)\left[\left(\Dv(\rho\u\otimes\u)\right)^{\varepsilon}-\Dv(\rho\u\otimes\u^{\varepsilon})\right]\u^{\varepsilon}\,dx\,dt
\\&=
-\int_0^T\int_{\T}\psi(t)\left[\left(\rho\u\otimes\u\right)^{\varepsilon}-(\rho\u\otimes\u^{\varepsilon})\right]\nabla\u^{\varepsilon}\,dx\,dt,
\end{split}
\end{equation*}
thus,
\begin{equation*}
\begin{split}|B_{\varepsilon}|&\leq
C\int_0^T\int_{\T}\left|\left(\rho\u\otimes\u\right)^{\varepsilon}-\rho(\u\otimes\u)^{\varepsilon}\right||\nabla\u^{\varepsilon}|\,dx\,dt
\\&\quad +
C\int_0^T\int_{\T}\left|\rho[(\u\otimes\u)^{\varepsilon}-\u\otimes\u^{\varepsilon}]\right||\nabla\u^{\varepsilon}|\,dx\,dt\\&
=: B_{1\varepsilon}+B_{2\varepsilon}.
\end{split}
\end{equation*}
Note that $\nabla\rho$ is bounded in $L^{\infty}(0,T;L^p(\T))$, thus we can apply Lemma \ref{key lemma} to
 control $B_{1\varepsilon}$ and $B_{2\varepsilon}$ as follows
 \begin{equation*}
 B_{1\varepsilon}\leq C\|\u\|_{L^q(0,T;B_p^{\alpha,\infty}( \T))}\|\nabla\rho\|_{L^p(0,T;L^p(\T))}\varepsilon^{3\alpha}\to 0
 \end{equation*}
and
 \begin{equation*}
 B_{2\varepsilon}\leq C\|\u\|_{L^q(0,T;B_p^{\alpha,\infty}( \T))}\varepsilon^{3\alpha-1}\to 0
 \end{equation*}
as $\varepsilon \to 0$ for $\alpha>\frac{1}{3}$. Thus, $B_{\varepsilon}$ tends to zero as $\varepsilon\to 0$.

\medskip

We are ready to pass to the limits in \eqref{the first two terms}. Letting $\varepsilon$ go to zero and using \eqref{convegence first term}-\eqref{the rest goes to zero}, what we have proved is that in the limit,
\begin{equation}
\label{limit equation first}
-\int_0^T\int_{\T}\psi_t\frac{1}{2}\rho|\u|^2\,dx\,dt=0
\end{equation}
for any test function $\psi\in \mathfrak{D}(0,\infty).$

\medskip

To obtain the energy conservation up to the initial time, we need to extend \eqref{limit equation first} for the test function $\psi(t)\in \mathfrak{D}(-1,\infty).$ To this end, it is necessary for us to have the continuity of $(\sqrt{\rho}\u)(t)$ in the strong topology at $t=0$.
%Adopting the similar argument in Vasseur-Yu \cite{VY} and in \cite{Yu2}, what we expected can be done.\\
Therefore we compute
\begin{equation}
\begin{split}
\label{control of sqrt density and u}
\text{ess} & \limsup_{t\to0}\int_{\T}|\sqrt{\rho}\u-\sqrt{\rho_0}\u_0|^2\,dx \\
&\leq\text{ess}\limsup_{t\to0}\left(\int_{\T}\rho|\u|^2\,dx-\int_{\T}\rho_0|\u_0|^2\,dx\right)
\\&\ \ +\text{ess}\limsup_{t\to0}\left(2\int_{\T}\sqrt{\rho_0}\u_0(\sqrt{\rho_0}\u_0-\sqrt{\rho}\u)\,dx\right).
\end{split}
\end{equation}

This yields
\begin{equation}
\begin{split}
\label{the rest term for initial time}
&\text{ess}\limsup_{t\to0}\int_{\T}|\sqrt{\rho}\u-\sqrt{\rho_0}\u_0|^2\,dx
\\&\ \ \ \leq 2\text{ess}\limsup_{t\to0}\int_{\T}\sqrt{\rho_0}\u_0(\sqrt{\rho_0}\u_0-\sqrt{\rho}\u)\,dx =: W.
\end{split}
\end{equation}

To show the continuity of $(\sqrt{\rho}\u)(t)$ in the strong topology at $t=0$, we need $W=0$. To this end, for any fixed $\phi \in \mathfrak{D}(\T)$ satisfying $\Dv \phi = 0$, we define the function $f$ on $[0, T]$ as
\begin{equation*}
f(t) = \int_{\T}(\rho\u)(t,x)\cdot\phi(x)\,dx.
\end{equation*}
Note from \eqref{condition for first theorem} that the function
$$
\int_{\T}(\rho\u)(t,x)\cdot\phi(x)\,dx
$$
is continuous function with respect to $t\in[0,T].$ On the other hand, note from \eqref{condition for first theorem} and \eqref{condition for kinetic energy} that
$$\rho\in L^{\infty}(0,T;L^{\infty}(\T))\quad\text{ and }\sqrt{\rho}\u \in L^{\infty}(0,T; L^2(\T)),$$
and hence we obtain that
$$\rho\u\in L^{\infty}(0,T;L^{2}(\T)).$$
From equation \eqref{Euler} we further know that
\begin{equation}
\label{integral for mass}
\frac{d}{dt}\int_{\T}(\rho\u)(t,x)\cdot\phi(x)\,dx=
\int_{\T}\rho\u\otimes\u:\nabla \phi\,dx,
\end{equation}
which is bounded due to \eqref{condition for kinetic energy}. Therefore from \cite[Corollary 2.1]{F04} it follows that
\begin{equation}
\label{weak continuous for mass}
\rho\u\in C([0,T]; L^{2}_{\text{weak}}(\T)).
\end{equation}
Applying the same argument to $\rho_t=-\Dv(\rho\u)$ with $\Dv\u=0$, we have \begin{equation}
\label{weak continuity for half density}
\sqrt{\rho}\in C([0,T];L_{\text{weak}}^{b}(\T))
\end{equation}
for any $b>1.$

We now estimate $W$ as follows
\begin{equation}\begin{split}\label{show B is zero}
W & :=2\text{ess}\limsup_{t\to0}\int_{\T}\u_0(\rho_0\u_0-\sqrt{\rho_0\rho}\u)\,dx
\\&\ \leq 2\text{ess}\limsup_{t\to0}\int_{\T}\u_0(\rho_0\u_0-\rho\u)\,dx \\
& \quad + 2\text{ess}\limsup_{t\to0}\int_{\T}\u_0(\sqrt{\rho}-\sqrt{\rho_0})\sqrt{\rho}\u\,dx.
\end{split}
\end{equation}

Meanwhile,
by \eqref{condition for first theorem}, we have \begin{equation}
\label{ssss}\sqrt{\rho}\u\in L^q(0,T;L^q(\T))\quad\text{ for any }q>3.
\end{equation}
Using \eqref{weak continuous for mass}, \eqref{weak continuity for half density} and \eqref{ssss}  in \eqref{show B is zero},  one deduces $W=0$ provided that $\u_0\in L^2(\T).$
 Thus, we have
\begin{equation*}
\text{ess}\limsup_{t\to0}\int_{\T}|\sqrt{\rho}\u-\sqrt{\rho_0}\u_0|^2\,dx=0,
\end{equation*}
which gives us
\begin{equation}
\label{continuity in strong topology 2}\sqrt{\rho}\u\in C([0,T];L^2(\T)).
\end{equation}

 From \eqref{continuity in strong topology 2}, we deduce
 \begin{equation}
 \label{limit000}\lim_{\tau\to0}\frac{1}{\tau}\int_0^{\tau}\int_{\T}\frac{1}{2}\rho|\u|^2\,dx\,dt=\int_{\T}\frac{1}{2}\rho_0|\u_0|^2\,dx.
 \end{equation}

 \medskip

 For each $\tau>0$ we choose a $C^1$ test function $\psi_{\tau,K}$ for \eqref{limit equation first} such that
 \begin{equation*}
 \begin{split}
 \psi_{\tau, K}(t)= \left\{\begin{array}{ll}
 \displaystyle \psi(t)\;\; & \text{ for } \,t\geq \tau+\frac{1}{K},\\\\
 \displaystyle \frac{t}{\tau}\;\; & \text{ for } t\leq \tau
 \end{array}\right.
 \end{split}
 \end{equation*}
 for $K>0$.
 %and $\psi_\tau$ is a $C^1$ function.
 Then replacing $\psi$ by $\psi_{\tau,K}$ in \eqref{limit equation first} we get
 \begin{equation}
\begin{split}
\label{AASS}
-\int_{\tau+\frac{1}{K}}^T & \int_{\T}  \psi_t\frac{1}{2}\rho|\u|^2\,dx\,dt-\int_{\tau}^{\tau+\frac{1}{K}}\int_{\T} (\partial_t\psi_{\tau,K}) \frac{1}{2}\rho|\u|^2\,dx\,dt \\
& =\frac{1}{\tau}\int_0^{\tau}\int_{\T}\frac{1}{2}\rho|\u|^2\,dx\,dt.
\end{split}
\end{equation}
Note that, \begin{equation*}
\begin{split}&
\left|\int_{\tau}^{\tau+\frac{1}{K}}\int_{\T}(\partial_t \psi_{\tau,K}) \frac{1}{2}\rho|\u|^2\,dx\,dt\right|\leq \frac{C}{K}\int_{\T}\frac{1}{2}\rho|\u|^2\,dx
\to0
\end{split}
\end{equation*}
as $K$ goes to infinity. This way by letting $K\to\infty,$ from \eqref{AASS} we derive
\begin{equation}
\begin{split}
\label{key limit-aaa}
-\int_{\tau}^T\int_{\T}\psi_t\frac{1}{2}\rho|\u|^2\,dx\,dt
=\frac{1}{\tau}\int_0^{\tau}\int_{\T}\frac{1}{2}\rho|\u|^2\,dx\,dt.
\end{split}
\end{equation}
%By \eqref{half-continuity for density}, \eqref{continuity in strong topology 2} and
By \eqref{limit000},
 passing into the limit as $\tau\to 0$ in \eqref{key limit-aaa}, one obtains
 \begin{equation}
\label{the second level limit}
\begin{split}
-\int_{\tau}^T\int_{\T}\psi_t\frac{1}{2}\rho|\u|^2\,dx\,dt
=\int_{\T}\frac{1}{2}\rho_0|\u_0|^2\,dx.
\end{split}
\end{equation}
For $\tilde t > 0$, taking
\begin{equation}\label{psi_name}
\psi(t) = \begin{cases} \displaystyle 0 \;\;\;\;\;\quad\quad\quad\quad\text{ if }t\leq \tilde{t}-\frac{\epsilon}{2},
%\\ \text{ is a  continuous function,}\text{ if } n\leq y\leq 2n,
\\\\ \displaystyle \frac{1}{2}+\frac{t-\tilde{t}}{\epsilon}\;\;\;\;\;\;\;\text{ if } \tilde{t}-\frac{\epsilon}{2}\leq t\leq \tilde{t}+\frac{\epsilon}{2},
\\\\ \displaystyle 1\,\;\;\;\;\quad\quad\;\quad\quad\text{ if }  t\geq \tilde{t}+\frac{\epsilon}{2},
\end{cases}\end{equation}
then \eqref{the second level limit} gives for every $\tilde{t}\geq \frac{\epsilon}{2}$ that
\begin{equation}
\label{the third level limit}
\begin{split}&
\frac{1}{\epsilon}\int_{\tilde{t}-\frac{\epsilon}{2}}^{\tilde{t}+\frac{\epsilon}{2}}\left(\int_{\T}\frac{1}{2}\rho|\u|^2\,dx\right)\,dt
=\int_{\T}\frac{1}{2}\rho_0|\u_0|^2\,dx.
\end{split}
\end{equation}

Applying the Lebesgue point Theorem, \eqref{the third level limit} implies that
\begin{equation*}
\begin{split}&\int_{\T}\frac{1}{2}\rho|\u|^2(t)\,dx
=\int_{\T}\frac{1}{2}\rho_0|\u_0|^2\,dx
\end{split}
\end{equation*}
for any $t\in[0,T]$, completing the proof of Theorem \ref{thm_main1}.
%This ends our proof of Theorem \ref{main result}.\\

\vskip0.3cm
\section{Proof of Theorem \ref{thm_main2}}
Following the previous section, we have
\begin{equation}
\label{the energy in distributation}
-\int_0^T\int_{\T}\psi_t\frac{1}{2}\rho^{\varepsilon}|\u^{\varepsilon}|^2\,dx\,dt+A_{\varepsilon}+B_{\varepsilon}=0,
\end{equation}
with
\begin{equation}
\label{Error 1}
A_{\varepsilon}=
\int_0^T\int_{\T}\psi(t)\left(\left((\rho\u)_t\right)^{\varepsilon}-(\rho^{\varepsilon}\u^{\varepsilon})_t\right)\u^{\varepsilon}\,dx\,dt
\end{equation}
and
\begin{equation}
\label{error 2} B_{\varepsilon}=
\int_0^T\int_{\T}\psi(t)\left(\left(\Dv(\rho\u\otimes\u)\right)^{\varepsilon}-\Dv(\rho\u\otimes\u^{\varepsilon})\right)\u^{\varepsilon}\,dx\,dt.
\end{equation}
To prove Theorem \ref{thm_main2}, we need to show $$A_{\varepsilon}\to 0,\quad\text{ and }\,B_{\varepsilon}\to 0\;\;\;\text{ as } \varepsilon\to 0.$$
% and the rest part is same to Section 3.
 We handle the term $A_{\varepsilon}$ as follows
 \begin{equation*}\begin{split}
A_{\varepsilon}&=
-\int_0^T\int_{\T}\psi_t\left(\left(\rho\u\right)^{\varepsilon}-\rho^{\varepsilon}\u^{\varepsilon}\right)\u^{\varepsilon}\,dx\,dt
 -\int_0^T\int_{\T}\psi\left(\left(\rho\u\right)^{\varepsilon}-\rho^{\varepsilon}\u^{\varepsilon}\right)\u_t^{\varepsilon}\,dx\,dt
\\&=: A_{1\varepsilon}+A_{2\varepsilon}.
\end{split}
\end{equation*}
The first term can be estimated as
\begin{equation*}
\begin{split}
&|A_{1\varepsilon}|\leq C(\psi_t)\int_0^T\int_{\T}\left|\left(\rho\u\right)^{\varepsilon}-\rho^{\varepsilon}\u\right||\u^{\varepsilon}|\,dx\,dt
+C(\psi_t)\int_0^T\int_{\T}\left|\rho^{\varepsilon}\u-\rho^{\varepsilon}\u^{\varepsilon}\right||\u^{\varepsilon}|\,dx\,dt.
\end{split}
\end{equation*}
We can then use Lemma \ref{key lemma} to  control the first term of the right-hand side of the above as follows
\begin{equation*}
\begin{split}&
\int_0^T\int_{\T}\left|\left(\rho\u\right)^{\varepsilon}-\rho^{\varepsilon}\u\right||\u^{\varepsilon}|\,dx\,dt
 \leq C \varepsilon^{\alpha} \|\rho\|_{L^{\infty}(0,T;L^{\infty}(\T))}\|\u\|^2_{B_p^{\beta,\infty}(0,T;B_q^{\alpha,\infty}(\T))}\to 0;
\end{split}
\end{equation*}
and the second term as
 \begin{equation*}
 \begin{split}
 \int_0^T\int_{\T} & \left|\rho ^{\varepsilon}\u-\rho^{\varepsilon}\u^{\varepsilon}\right||\u^{\varepsilon}|\,dx\,dt
 \\&\leq C \|\rho^{\varepsilon}\|_{L^{\infty}(0,T;L^\infty(\T))}\|\u^{\varepsilon}-\u\|_{L^p(0,T;B_q^{\alpha,\infty}(\T))}\|\u^{\varepsilon}\|_{L^p(0,T;B_q^{\alpha,\infty}(\T))}
 \\&\leq C \varepsilon^{\alpha} \|\rho\|_{L^{\infty}(0,T;L^{\infty}(\T))}\|\u\|^2_{L^p(0,T;B_q^{\alpha,\infty}(\T))}
 \to 0
 \end{split}
 \end{equation*}
 as $\varepsilon\to 0,$
where $p,q\geq 2.$

Similar argument applying to $A_{2\varepsilon}$ yields
\begin{equation*}
 \begin{split}
 |A_{2\varepsilon}| & \leq C\int_0^T\int_{\T}\left|\left(\rho\u\right)^{\varepsilon}-\rho^{\varepsilon}\u^{\varepsilon}\right||\u_t^{\varepsilon}|\,dx\,dt
 \\&\leq  C\int_0^T\int_{\T}\left|\left(\rho\u\right)^{\varepsilon}-\rho^{\varepsilon}\u\right||\u_t^{\varepsilon}|\,dx\,dt
 +C\int_0^T\int_{\T}\left|\left(\rho^{\varepsilon}\u\right)-\rho^{\varepsilon}\u^{\varepsilon}\right||\u_t^{\varepsilon}|\,dx\,dt
 \\&\leq 2C\|\rho\|_{L^{\infty}(0,T;L^{\infty}(\T))}\|\u\|_{B_p^{\beta,\infty}(0,T;B_q^{\alpha,\infty}(\T))}\varepsilon^{\alpha+2\beta-1},
 %\\&+C \|\rho^{\varepsilon}-\rho\|_{L^{\infty}(0,T;L^r(\T))}\|\u\|^2_{B_p^{\alpha,\infty}(0,T;B_q^{\beta,\infty}(\T))},
 \end{split}
 \end{equation*}
 which converges to zero as $\varepsilon$ goes to zero when $\alpha+2\beta>1.$
 Thus, we have $A_{\varepsilon}\to 0$ as $\varepsilon$ goes to zero.
 %We are able to apply the same argument to show that $B_{\varepsilon}\to 0$ as $\varepsilon$ goes to zero.

To handle $B_{\varepsilon}$,
 \begin{equation*}
 \begin{split}| B_{\varepsilon}|&\leq C
\int_0^T\int_{\T}\left|(\rho\u\otimes\u)^{\varepsilon}-(\rho\u)^{\varepsilon}\otimes\u^{\varepsilon}\right||\nabla\u^{\varepsilon}|\,dx\,dt
\\&\leq  C
\int_0^T\int_{\T}\left|(\rho\u\otimes\u)^{\varepsilon}-(\rho\u)^{\varepsilon}\otimes\u\right||\nabla\u^{\varepsilon}|\,dx\,dt
\\&\quad+
 C
\int_0^T\int_{\T}\left|(\rho\u)^{\varepsilon}\otimes \u-(\rho\u)^{\varepsilon}\otimes\u^{\varepsilon}\right||\nabla\u^{\varepsilon}|\,dx\,dt
\\&=: B_{1\varepsilon}+B_{2 \varepsilon}.
\end{split}
\end{equation*}
Note that $\u \in B_p^{\beta,\infty}(0,T;B_q^{\alpha,\infty}(\T))$. From Lemma \ref{key lemma} we conclude that
$$B_{1\varepsilon}\leq C \|\u\|^3_{B_p^{\beta,\infty}(0,T;B_q^{\alpha,\infty}(\T))}\varepsilon ^{2\alpha+\beta-1}\to 0, $$
and
$$B_{2\varepsilon}\leq C \|\u\|^3_{B_p^{\beta,\infty}(0,T;B_q^{\alpha,\infty}(\T))}\varepsilon ^{2\alpha+\beta-1}\to 0, $$
for any $2\alpha+\beta>1$ and $p, q \ge 3$.

Letting $\varepsilon\to 0$ in \eqref{the energy in distributation}, one obtains
\begin{equation*}
-\int_0^T\int_{\T}\psi_t\frac{1}{2}\rho|\u|^2\,dx\,dt=0.
\end{equation*}
Because the regularity of $\rho$ and $\u$ allow us to have
$$\sqrt{\rho}\in C([0,T]; L^2_{\text{weak}}(\T))$$
and $$\sqrt{\rho}\u\in C([0,T]; L^2(\T)).$$
Thus, we can repeat the same argument in Section 3 to show
 \begin{equation*}
\begin{split}&\int_{\T}\frac{1}{2}\rho|\u|^2(t)\,dx
=\int_{\T}\frac{1}{2}\rho_0|\u_0|^2\,dx
\end{split}
\end{equation*}
for any $t\in[0,T].$

\section*{Acknowledgements}
The work of Chen is partially supported by National Science Foundation under Grant DMS-1613375. This material is based upon work supported by the National Science Foundation under Grant No. DMS-1439786 while the first author was in residence at the Institute for Computational and Experimental Research in Mathematics in Providence, RI, during the Spring 2017 semester.

The authors are also grateful to Prof. Claude Bardos for valuable suggestions.


\bigskip\bigskip
\begin{thebibliography}{99}




\bibitem{BDIS}T. Buckmaster, C. De Lellis, P. Isett, and L. Sz\'ekelyhidi, Jr., \emph{Anomalous dissipation for 1/5-Holder Euler flows.} Ann. of Math. {\bf 182} (2015), 127-172.

\bibitem{BDS} T. Buckmaster, C. De Lellis, and L. Sz\'ekelyhidi, Jr., \emph{ Dissipative Euler flows with Onsager-critical spatial regularity.} Comm. Pure and Appl. Math. {\bf 69} (2016), 1613-1670.

\bibitem{BDSV}
T. Buckmaster, C. De Lellis, L. Sz\'ekelyhidi, Jr., and V. Vicol, \emph{Onsager's conjecture for admissible weak solutions.} Preprint 2017. arXiv:1701.08678.

\bibitem{CCFS} A. Cheskidov, P. Constantin, S. Friedlander, and R. Shvydkoy, \emph{ Energy conservation and Onsager's conjecture for the Euler equations.} Nonlinearity {\bf 21} (2008) 1233-1252.

\bibitem{CET} P. Constantin, W. E and E. Titi, \emph{Onsager's conjecture on the energy conservation for solutions of Euler's equation.}
Comm. Math. Phys. {\bf 165} (1994), no. 1, 207-209.

\bibitem{DS}
C. De Lellis, and L. Sz\'ekelyhidi, Jr., \emph{The Euler equations as a differential inclusion.} Ann. of Math. {\bf 170} (2009), 1417-1436.

\bibitem{DS13}
C. De Lellis, and L. Sz\'ekelyhidi, Jr., \emph{Continuous dissipative Euler flows and a conjecture
of Onsager.} European Congress of Mathematics, Eur. Math. Soc., Z\"urich (2013), 13-29.

\bibitem{DS13b}
C. De Lellis, and L. Sz\'ekelyhidi, Jr., \emph{Dissipative continuous Euler flows.} Invent. Math. {\bf 193} (2013), 377-407.

\bibitem{DS14}
C. De Lellis, and L. Sz\'ekelyhidi, Jr., \emph{Dissipative Euler flows and OnsagerÕs conjecture.} J. Eur. Math. Soc. (JEMS) {\bf 16} (2014), 1467-1505.

\bibitem{DR}
J. Duchon and R. Robert, \emph{Inertial energy dissipation for weak solutions of incompressible
Euler and Navier-Stokes equations.} Nonlinearity {\bf 13} (2000), 249-255.

\bibitem{Evans} L. C. Evans, \emph{Partial differential equations}. Second edition, Graduate Studies in Mathematics, 19. American Mathematical Society, Providence, RI, 2010.

\bibitem{Ey} G. L. Eyink. \emph{Energy dissipation without viscosity in ideal hydrodynamics: I. Fourier analysis and local energy
transfer.} Phys. D {\bf 78} (1994), 222-240.

\bibitem{ES06}
G. L. Eyink and K. R. Sreenivasan, \emph{Onsager and the theory of hydrodynamic turbulence.} Reviews of Modern Physics {\bf 78} (2006), 1-46.

\bibitem{FNP} E. Feireisl, A. Novotn\'{y}, H. Petzeltov\'{a}, \emph{ On the
existence of globally defined weak solutions to the Navier-Stokes
equations.} J. Math. Fluid Mech. \textbf{3} (2001), 358-392.


\bibitem{F04}E. Feireisl, \emph{Dynamics of viscous compressible fluids.} Oxford Lecture Series in Mathematics and its Applications, 26. Oxford Science Publications. The Clarendon Press, Oxford University Press, New York, 2004.


\bibitem{FGSW}
E. Feireisl, P. Gwiazda, A. Swierczewska-Gwiazda, and E. Wiedemann,
\emph{Regularity and energy conservation for the compressible Euler equations.} Arch. Ration. Mech. Anal. {\bf 223} (2017), 1375-1395.

\bibitem{F95}
U. Frisch, \emph{Turbulence: The legacy of A. N. Kolmogorov.} Cambridge University Press, Cambridge, 1995.


%\bibitem{H}E. Hopf,\emph{\"{U}ber die Anfangswertaufgabe f\"{u}r die hydrodynamischen Grundgleichungen, }Math. Nachr. 4 (1951), 213-231.

\bibitem{I}P. Isett, \emph{A proof of Onsager's conjecture.} Preprint, 2016. arXiv:1608.08301.

\bibitem{I2}
P. Isett, \emph{On the endpoint regularity in OnsagerÕs conjecture}. Preprint, 2017.  arXiv:1706.01549.


%\bibitem{Le} J. Leray, \emph{ Sur le mouvement d'un liquide visqueux emplissant l'espace.} (French) Acta Math. 63 (1934), no. 1, 193-248.
\bibitem{LS} T.M. Leslie, R. Shvydkoy, \emph{The energy balance relation for weak solutions of the density-dependent Navier-Stokes equations.} J. Differ. Equ. {\bf 261} (2016), 3719-3733.

%\bibitem{Le} J. Leray, \emph{ Sur le mouvement d'un liquide visqueux emplissant l'espace.} (French) Acta Math. 63 (1934), no. 1, 193-248.
\bibitem{LV}
I. Lacroix-Violet, A. Vasseur,\emph{Global weak solutions to the compressible quantum Navier-Stokes equation and its semi-classical limit}. Preprint, 2016. arXiv:1607.06646.

\bibitem{L} P.-L. Lions, \emph{
Mathematical topics in fluid mechanics. Vol. 1.
Incompressible models.} Oxford Lecture Series in Mathematics and its Applications, 3. Oxford Science Publications. The Clarendon Press, Oxford University Press, New York, 1996.

\bibitem{Lions} P.-L. Lions, \emph{ Mathematical topics in fluid mechanics. } Vol. 2. Compressible models. Oxford Lecture Series in Mathematics and its Applications, 10. Oxford Science Publications. The Clarendon Press, Oxford University Press, New York, 1998.

%\bibitem{MV} A. Mellet, A. Vasseur, \emph{ On the barotropic compressible Navier-Stokes equations}. Comm. Partial Differential Equations {\bf 32} (2007), 431-452.

\bibitem{O} L. Onsager, \emph{Statistical Hydrodynamics.} Nuovo Cimento (Supplemento), {\bf 6}  (1949), 279-287.

\bibitem{R03}
R. Robert, \emph{Statistical hydrodynamics (Onsager revisited).} Handbook of Mathematical Fluid Dynamics vol 2 ed Friedlander and Serre, Elsevier, Amsterdam, (2003), 1-55.

\bibitem{Scheffer}
V. Scheffer, \emph{An inviscid flow with compact support in space-time.} J. Geom. Anal. {\bf 3} (1993), 343-401.

\bibitem{Serrin} J.
Serrin,
\emph{The initial value problem for the Navier-Stokes equations.} 1963 Nonlinear Problems (Proc. Sympos., Madison, Wis., 1962) pp. 69-98 Univ. of Wisconsin Press, Madison, Wis.

\bibitem{Shinbrot} M.
Shinbrot,\emph{
The energy equation for the Navier-Stokes system.}
SIAM J. Math. Anal. {\bf 5} (1974), 948-954.

\bibitem{Shnirelman}
A. Shnirelman, \emph{On the nonuniqueness of weak solution of the Euler equation.} Commun. Pure Appl. Math. {\bf 50} (1997), 1261-1286.


\bibitem{VY} A. Vasseur, C. Yu, \emph{Existence of global weak solutions for 3D degenerate compressible Navier-Stokes equations.}  Invent. Math. 206 (2016), no. 3, 935-974.
%\bibitem{Yu} C.Yu, \emph{A new proof of the energy conservation for the Navier-Stokes equations}. Preprint, 2016, 	arXiv:1604.05697.
\bibitem{Yu2} C.Yu, \emph{Energy conservation for the weak solutions of the compressible Navier-Stokes equations.} to appear in  Arch. Ration. Mech. Anal.
\end{thebibliography}
\end{document}